\documentclass[12pt,a4paper]{article}
\usepackage{amssymb,amsmath,amsthm}
\usepackage[english]{babel}
\usepackage{t1enc}
\usepackage[latin2]{inputenc}
\usepackage{epsfig}
\usepackage{subfig}
\usepackage{epstopdf}
\usepackage{xcolor}
\usepackage{hyperref}

\newtheorem{thm}{Theorem}[section]
\newtheorem{cor}[thm]{Corollary}
\newtheorem{defi}[thm]{Definition}
\newtheorem{lem}[thm]{Lemma}

\newtheorem{claim}[thm]{Claim}

\theoremstyle{remark}

\usepackage{indentfirst}
\def\C{\mathcal {C}}
\def\FF{\mathcal {F}}

\def\red1{\color{red}1\color{black}}
\def\blue1{\color{blue}1\color{black}}
\def\mabab{\begin{bmatrix}
0 & 1 & 0\\
1 & 0 & 1
\end{bmatrix}}

\makeatletter
\def\blfootnote{\gdef\@thefnmark{}\@footnotetext}
\makeatother

\begin{document}

\title{On dual-ABAB-free and related hypergraphs}
	%Coloring dual ABAB-free hypergraphs and related problems}
\author{Bal\'azs Keszegh\thanks{HUN-REN Alfréd Rényi Institute of Mathematics and ELTE Eötvös Loránd University, Budapest, Hungary. 
		Research supported by the J\'anos Bolyai Research Scholarship of the Hungarian Academy of Sciences, by the National Research, Development and Innovation Office -- NKFIH under the grant K 132696 and FK 132060, by the \'UNKP-23-5 New National Excellence Program of the Ministry for Innovation and Technology from the source of the National Research, Development and Innovation Fund and by the ERC Advanced Grant ``ERMiD''. This research has been implemented with the support provided by the Ministry of Innovation and Technology of Hungary from the National Research, Development and Innovation Fund, financed under the  ELTE TKP 2021-NKTA-62 funding scheme.}\\
%\small Alfr\'ed R{\'e}nyi Institute of Mathematics, \\ \small Hungarian Academy of Sciences\\ \small Re\'altanoda u. 13-15 Budapest, 1053 Hungary
\and D\"om\"ot\"or P\'alv\"olgyi$^*$
}

\maketitle

%\keywords{geometric hypergraph, pseudo halfplane, polychromatic coloring, cover-decomposition, hitting set}

\begin{abstract}
	Geometric motivations warranted the study of hypergraphs on ordered vertices that have no pair of hyperedges that induce an alternation of some given length. Such hypergraphs are called ABA-free, ABAB-free and so on. Since then various coloring and other combinatorial results were proved about these families of hypergraphs. We prove a characterization in terms of their incidence matrices which avoids using the ordering of the vertices. Using this characterization, we prove new results about the dual hypergraphs of ABAB-free hypergraphs. In particular, we show that dual-ABAB-free hypergraphs are not always proper $2$-colorable even if we restrict ourselves to hyperedges that are larger than some parameter $m$.
\end{abstract}

%\newpage

\section{Introduction}

For an integer $t$, we denote by $(AB)^{t/2}$ the alternating sequence of letters $A$ and $B$ of length $t$. For example, $(AB)^{1.5}=ABA$ and $(AB)^2=ABAB$.
Next we define what we mean by a $(AB)^{t/2}$-free hypergraph.

\begin{defi}
	\mbox{}
	\begin{enumerate}
		\item Two subsets $A,B$ of an ordered set of elements form an \emph{$(AB)^{t/2}$-sequence} if there are $t$ elements $a_1 < b_1 < a_2 < b_2 < \ldots$ such that $\{a_1,a_2,\ldots\} \subseteq A \setminus B$ and  $\{b_1,b_2,\ldots\} \subseteq B \setminus A$.
		\item A hypergraph with an ordered vertex set is {\em $(AB)^{t/2}$-free} if it does not contain two hyperedges $A$ and $B$ that form an $(AB)^{t/2}$-sequence.
		\item A hypergraph with an unordered vertex set is $(AB)^{t/2}$-free if there is an ordering of its vertices such that the hypergraph with this ordered vertex set is $(AB)^{t/2}$-free.
		%\item The family of all $(AB)^{t/2}$-free hypergraphs is denoted by $\mathcal{(AB)}^t\emph{-free}$.
		\item A hypergraph on an unordered vertex set is dual-$(AB)^{t/2}$-free if its dual hypergraph\footnote{The dual of a hypergraph is the hypergraph that we obtain by reversing the roles of vertices and hyperedges.} is $(AB)^{t/2}$-free.
	\end{enumerate}
\end{defi}

\subsection{Geometric motivation}
ABA-free hypergraphs and its generalizations were first studied in \cite{abafree} motivated by the fact that such hypergraphs can be realized by natural planar regions. 
We say that a collection of points $P$ and regions $R$ realize $H=(V,E)$ if there are bijections $V\to P$ and $E\to R$ that preserve incidence.
We say that two curves cross $t$ times if they have finitely many intersection points and exactly $t$ intersection points cannot be eliminated by local perturbation, i.e., they are not just touchings.
A curve is $x$-monotone if it is the part of a graph of a function.
An upward region is a part of the plane that lies above the graph of a function.
Two upward regions are $t$-intersecting if their boundaries intersect at most $t$ times.
In \cite{abab}, generalizing the $t=1$ result of \cite{abafree}, it was shown that $(AB)^{(t+2)/2}$-free hypergraphs are equivalent to hypergraphs realizable by such regions.

\begin{lem}[\cite{abab}]\label{lem:parabolas2ABAB}
	Let $t \ge 1$ be an integer and let $\C$ be a family of $t$-intersecting upwards regions in the plane and let $S$ be a finite set of points. Then $H(S,\C)$ is an $(AB)^{(t+2)/2}$-free hypergraph.
\end{lem}

\begin{lem}[\cite{abab}]\label{lem:ABAB2parabolas}
	Let $t \ge 1$ be an integer and let $H$ be a finite $(AB)^{(t+2)/2}$-free hypergraph. Then there is a family $\C$ of $t$-intersecting upwards regions in the plane and a finite point set $S$ such that $H(S,\C)$ is isomorphic to $H$.
\end{lem}

Without going into every detail, we list some further relations.
In \cite{abafree}, further variants of ABA-free hypergraphs were investigated, among others it was shown that for ABA-free hypergraphs $x$-monotonicity is not necessary to be assumed in the above results. Also, if we take in addition the complements of the hyperedges of an ABA-free hypergraph, then we get hypergraphs realizable by pseudohalfplanes. In \cite{abafree,abab}, it was shown that ABAB-free hypergraphs are equivalent to the hypergraphs realizable by pseudoparabolas and also by stabbed pseudodisks (i.e., a family of pseudodisks, all containing the origin).

\subsection{Incidence matrices}

The incidence matrix $M(H)$ of a hypergraph $H$ is the $0$--$1$ matrix obtained by associating a row to each vertex and a column to each hyperedge.
$M_{ve}$ is defined as $1$ if $v\in e$ and as $0$ if $v\notin e$.
If the vertices (and/or the hyperedges) of the hypergraph are ordered, then we require that the rows (and/or columns) of the incidence matrix are ordered the same way.
For two matrices $M$ and $P$ we say that $M$ contains $P$ if it has a subset of rows and columns that form a submatrix (in the same order) that is \emph{exactly the same} as $P$.

$M$ is $P$-free if it does not contain $P$. Be careful that this definition is different from the usual definition used in extremal $0-1$ matrix theory, where $1$'s can be changed to $0$'s to get $P$. Our definition is based on the definition of induced subgraphs of graphs while the other one is based on the definition of (not necessarily induced) subgraphs of graphs.

In \cite{abafree} equivalent definitions of being ABA-free were defined using the incidence matrix of the hypergraph. We generalize this to $(AB)^t$-free hypergraphs, which will be an essential tool in the proofs of our results about colorings.

Let $X_t^T=\begin{bmatrix}
	0 & 1 & 0 & ... \\
	1 & 0 & 1 & ...
\end{bmatrix}$ with $t$ columns and $2$ rows that contain alternately $\begin{bmatrix} 0\\ 1 \end{bmatrix}$ and $\begin{bmatrix} 1\\ 0 \end{bmatrix}$ and $X_t$ is the transpose of $X_t^T$ with $2$ columns and $t$ rows.
We denote by $X_t'$ the matrix obtained by swapping the $2$ columns of $X_t$.

\begin{thm}\label{thm:matrixequi}
	Given a hypergraph $H$, the following are equivalent for every integer $t$:
	\begin{itemize}
		\item[(a)] $H$ is an $(AB)^{t/2}$-free hypergraph,
		\item[(b)] there is a permutation of the rows of $M(H)$ such that the matrix becomes  
		$X_{t}$-free and $X_{t}'$-free,
		\item[(c)] there is a permutation of the rows and columns of $M(H)$ such that the matrix becomes  
		$X_{t-1}$-free.
	\end{itemize}
\end{thm}

%Ez már szerepel korábban footnoteban Taking the dual of the hypergraph means taking the transpose of its incidence matrix, thus we get the following:

\begin{cor}\label{cor:matrixequi}
	$\cal H$ is dual-$(AB)^{t/2}$-free if and only if its incidence matrix has a row and column ordering which is $X_{t-1}^T$-free.
\end{cor}

In particular, $\cal H$ is dual-ABAB-free if and only if its incidence matrix has a row and column ordering which is $\begin{bmatrix}
		0 & 1 & 0\\
		1 & 0 & 1
	\end{bmatrix}$-free.

Theorem \ref{thm:matrixequi} has also the following nice consequence, using a result of Kor\'andi, Pach and Tomon \cite{korandipachtomon}.
We give the details of the implication later.

\begin{cor}\label{cor:homogeneous}
	For every integer $t$ there exists a constant $c$ such that if $H$ is an $(AB)^{t/2}$-free hypergraph on $n$ vertices and with $n$ hyperedges then there exist a subset $V_0$ of its vertices and a subset $E_0$ of its hyperedges such that $|V_0|,|E_0|\ge cn$ and either $v\in e$ for all $v\in V_0, e\in E_0$ or $v\notin e$ for all $v\in V_0, e\in E_0$.
\end{cor}

Note that due to the symmetry on vertices and hyperedges in the conclusion, the same holds for dual-$(AB)^t$-free hypergraphs.

For example, applying this theorem for ABAB-free hypergraphs implies that given $n$ points in the plane and $n$ stabbed pseudodisks, then there are $cn$ points and $cn$ disks such that either all of these pseudodisks contain all $cn$ points or none of them contains any of the points.

%Na ezt a trivialitást azért már tényleg ne mondjuk ki, ha nem is használjuk soha.
%We finish this section with a simple yet useful relation between a hypergraph family and its dual that holds in general.
%
%\begin{claim}\label{claim:subhyp}
%	Let $\FF$ be a family of hypergraphs and $\FF'$ be the family of the duals of the hypergraphs in $\FF$. If $\FF'\subseteq \FF$, then necessarily $\FF'=\FF$.
%\end{claim}

\subsection{Colorings}

The chromatic number $\chi$ of a hypergraph $H$ is the least number $c$ such that its vertices admit a \emph{proper $c$-coloring}, i.e., a coloring in which every hyperedge is non-monochromatic (i.e., contains two vertices with different colors).
Given a family of hypergraphs $\FF$, we denote by $\chi(\FF)$ the maximum of $\chi(H)$ over $H\in \FF$.

Motivated by cover-decomposability and conflict-free coloring problems of hypergraphs realizable by geometric objects, the following coloring problem is even more important in our setting.
Given a family of hypergraphs $\FF$ and a positive integer $c$, let $m(\FF,c)$ denote the least integer such that the vertices of every hypergraph $H\in\FF$ can be colored with $c$ colors such that every hyperedge of size at least $m(\FF,c)$ is non-monochromatic.
In other words, for every hypergraph $H\in\FF$ the sub-hypergraph of $H$ that consists of all the hyperedges of size at least $m(\FF,c)$ is proper $c$-colorable.
We denote by $\chi_m(\FF)$ the least integer $c$ for which such a finite $m(\FF,c)$ exists (otherwise, define $\chi_m(\FF)=\infty)$.
Investigating this parameter for geometrically realizable families has a huge literature, with results about translates and homothets of disks, convex polygons in the plane, pseudodisk families and many other families of regions. We refer the reader to the webpage \cite{cogezoo} for an up-to-date interactive database of results.

Note that vertex-coloring the dual of a hypergraph is equivalent to edge-coloring the original hypergraph, or in the geometric setting, to coloring the regions instead of the points. This motivated the study of coloring problems of dual hypergraphs. In particular, proper coloring dual-ABAB-free hypergraphs is equivalent to coloring stabbed pseudo-disks such that no point contained in at least two pseudo-disks is monochromatic. Similar equivalence holds for $\chi_m$.

ABA-free hypergraphs were first studied in \cite{abafree} where the authors generalized a polychromatic (i.e., where every color has to appear on every hyperedge) coloring result of Smorodinsky and Yuditsky \cite{smorodinsky-yuditsky} about halfplanes to pseudohalfplanes. In particular, in \cite{abafree} it was shown that for the family of ABA-free hypergraphs we have $\chi_m(\FF)=2$ and then in \cite{kbdiscretehelly} it was shown, among others, that we also have $\chi(\FF)=3$. As the dual of an ABA-free hypergraph is also ABA-free \cite{abafree}, these hold for the dual families too. About further properties of $ABA$-free hypergraphs, see also \cite{pseudoconvex}.

Continuing with ABAB-free hypergraphs, in \cite{abafree} it was shown that they are not always $2$-colorable. Then in \cite{abab}, it was shown that this is true even if we restrict ourselves to hyperedges of size at least $m$ for some fixed $m$. On the other hand, it was also shown that ABAB-free hypergraphs always admit a proper $3$-coloring. Summarizing, for the family of ABAB-free hypergraphs we have $\chi(\FF)=\chi_m(\FF)=3$.

In this paper, we mainly study dual-ABAB-free hypergraphs, about which not much was known earlier. For proper colorings of dual-ABAB-free hypergraphs, an upper bound of $4$ follows from earlier results \cite{psdiskwrtpsdisk} and here we show a simple construction that $4$ is possible, and so we establish the following.

\begin{claim}\label{claim:chi4}
	For every dual-ABAB-free hypergraph $\chi\le 4$, while there exist dual-ABAB-free hypergraphs with $\chi=4$.
\end{claim}

This implies that the $\chi_m$ of the family of dual-ABAB-free hypergraphs is also at most $4$. Our main coloring result shows that it is also larger than $2$.

\begin{thm}\label{thm:not2col}
	For every $m$ there exists an $m$-uniform dual-ABAB-free hypergraph which is not proper $2$-colorable.
\end{thm}

Thus, the $\chi_m$ of dual-ABAB-free hypergraphs is either $3$ or $4$; it remains an open problem to determine which one is the right answer.\footnote{We note that there are similar problems where deciding whether $\chi_m$ is $3$ or $4$ turned out to be much harder than proving that it is at least $3$. In particular, for hypergraphs realizable by disks it was only recently shown that $\chi_m=4$ \cite{damdom}, while for their duals it is still not known if the answer is $3$ or $4$ \cite{wcf2}.}

Continuing, for ABABA-free hypergraphs it was shown in \cite{abab} that $\chi=\chi_m=\infty$. This implies that for $(AB)^{t/2}$-free hypergraphs for any $t\ge 5$ we have $\chi=\chi_m=\infty$.

Finally, dual-ABABA-free hypergraphs and in general dual-$(AB)^{t/2}$-free hypergraphs with $t\ge 5$ were not investigated before and as of now it is not even known if $\chi_m$ is finite for these families.

\subsection{The maximal number of hyperedges}\label{sec:sizes}

We also consider bounding the maximal possible number of hyperedges in the considered hypergraphs. In this respect, we determine the exact orders of magnitude.

\begin{thm}\label{thm:ext}
	Let $n\ge t\ge 3$.\\
	The maximal size of an $(AB)^{t/2}$-free hypergraph on $n$ vertices is $\Theta(n^{t-1})$, where the hidden constant depends on $t$.\\
	The maximal size of a dual-$(AB)^{t/2}$-free hypergraph on $n$ vertices is $\Theta(tn^{2})$, where the hidden constant does not depend on $t$.
\end{thm}	

The proof of this theorem is quite standard, for similar methods, see for example \cite{sharir,romgunter}. For the primal $t/2$-uniform case for $t$ even, i.e., when we consider the maximal number of hyperedges only of size $t/2$ in an $(AB)^{t/2}$-free hypergraph, F\" uredi et al. \cite{furedi} showed the exact bound $\binom{n}{t/2}-\binom{n-t/2}{t/2}$. Note that counting hyperedges of smaller size is not interesting as they cannot ruin $(AB)^{t/2}$-freeness, so all of them can be in the hypergraph. For larger hyperedges, bounding the sum of the sizes of the hyperedges (i.e., the vertex-hyperedge incidences, same as the sum of the degrees) is another interesting question; for such results, see \cite{gunby}.

\subsection{Organization of the paper}	
In Section \ref{sec:proofs}, we prove the results stated above about the incidence matrices, about colorings and about the number of hyperedges. %In Section \ref{sec:sizes} we give bounds on the sizes of primal and dual hypergraphs that avoid an alternation of a given length.
In Section \ref{sec:shrinkability}, we consider two related problems, existence of shallow hitting sets and shrinkability and we prove some negative results about them. Finally, in Section \ref{sec:discussion}, we conclude with a discussion of further research directions.
	
\section{Proofs}\label{sec:proofs}
\begin{proof}[Proof of Theorem \ref{thm:matrixequi}.]
	First, ordering the vertices of the hypergraph corresponds to permuting the columns of its incidence matrix. Thus, the equivalance of $(a)$ and $(b)$ follows from the definition of $(AB)^t$-free hypergraphs. 
	
	To prove $(c)\rightarrow(b)$, suppose $(b)$ is false, i.e., that in any permutation of the rows of $M(H)$ there is an occurrence of one of the two matrices forbidden in $(b)$, $X_{t}$ and $X_{t}'$. In any permutation of $X_{t}$ or $X_{t}'$, we get back one of $X_{t}$ or $X_{t}'$, both of which contains a copy of $X_{t-1}$. Thus, by any permutation of the rows and columns of $M(H)$, we get a matrix that contains $X_{t-1}$. Thus, we can conclude that $\neg(b)\rightarrow \neg(c)$, which is the contrapositive of $(c)\rightarrow(b)$.

	Finally, we show $(b)\rightarrow(c)$. Take a permutation of rows such that it is $X_{t}$ and $X'_{t}$-free; call this matrix $M'(H)$. Notice that any ordering of the columns of $M'(H)$ gives an $X_{t}$-free and $X'_{t}$-free matrix. We define the following comparability relation on the columns of $M'(H)$: for two columns $A,B$ we have $A<B$ if and only if the first row where they differ has a $0$ in $A$ and a $1$ in $B$. It is easy to see that this gives a partial ordering of the columns, in fact, only identical columns are incomparable. By ordering the columns according to this relation, we get a matrix $M''(H)$, which is still  $X_{t}$-free and $X'_{t}$-free. We claim that $M''(H)$ is also $X_{t-1}$-free. Assume on the contrary that there exist a pair of columns $A<B$ in $M''(H)$ which contains a copy of $X_{t-1}$. As $A<B$, the first column where they differ has a $0$ in $A$ and a $1$ in $B$. Since this must be to the above of the copy of $X_{t-1}$, together they form an $X'_{t}$, so we get a contradiction.	
	%alternative pf: Finally, to prove $(a)\rightarrow(c)$, we can use Lemma \ref{lem:ABAB2parabolas} to get a representation with $(2t-2)$-intersecting curves in the plane. The rows should be ordered according to the $x$-coordinates of the points, while the columns can be ordered according to the $y$-coordinates of the curves intersected with a vertical line that precedes the first point. An $X_{2t-1}$ would mean that the curves intersect $2t-2$ times between the points corresponding to the rows of the submatrix, plus once more before the first point, as they started in the reverse order, a contradiction.		
	%alternative pf:
	%notice that for any pair of hyperedges, it is impossible that their respective columns contain both $X_{2t-1}$ and $X_{2t-1}'$, as then they would also contain $X_{2t}$ or $X_{2t}'$. Define $A<B$ if their columns contain $X_{2t-1}$ and $A>B$ if their columns contain $X_{2t-1}'$. It is easy to verify that this gives a directed acyclic graph, which we can arbitrarily extend to a complete order. Ordering the columns according to ``$<$'' we get a $X_{2t-1}$-free matrix, and thus $(b)\rightarrow (c)$.
\end{proof}

\begin{proof}[Proof of Corollary \ref{cor:homogeneous}.]
	Kor\'andi, Pach and Tomon \cite{korandipachtomon} have recently shown that if $P$ is a $2\times k$ $0-1$ matrix that does not contain a homogeneous (i.e., all-$0$ or all-$1$) $2\times 2$ submatrix, then any matrix $A$ that does not contain $P$ must  contain a homogeneous submatrix of size $cn\times cn$ for a constant $c$ depending only on $P$.	
	
	They applied this result for the incidence matrices of ABA-free hypergraphs, using the special case of Theorem \ref{thm:matrixequi} for ABA-free hypergraphs proved in \cite{abafree}. Having Theorem \ref{thm:matrixequi} in hand, we can now also apply it to $(AB)^{t}$-free hypergraphs for any $t$ as well to conclude the statement. We only need to check that $X_{t}$ does not contain a homogeneous $2\times 2$ submatrix, which trivially holds.
\end{proof}

%\begin{proof}[Proof of Claim \ref{claim:subhyp}.]
%	$\FF'\subseteq \FF$ means that for every $H\in \FF$ there exists a $G\in \FF$ such that the dual of $H$ is equal to $G$. 
%	
%	$\FF\subseteq \FF'$ on the other hand means that for every $H\in \FF$ there exists a $G\in \FF$ such that the dual of $G$ is equal to $H$.
%	
%	The two are equivalent, therefore $\FF'\subseteq \FF$ implies $\FF\subseteq \FF'$ and in turn $\FF'= \FF$, as claimed.	
%\end{proof}

\begin{proof}[Proof of Claim \ref{claim:chi4}.]
The chromatic number is at most $4$ for a much wider range of hypergraphs, that is, for intersection hypergraphs of pseudo-disks with respect to pseudo-disks \cite{psdiskwrtpsdisk} (whereas dual-ABAB-free hypergraphs correspond to hypergraphs of stabbed pseudo-disks with respect to points \cite{abab}). This implies the same upper bound for dual-ABAB-free hypergraphs.

For the lower bound it is enough to prove that $K_4$ is realizable as a dual-ABAB-free hypergraph. By Corollary \ref{cor:matrixequi}, it is enough to order the rows and colums of the incidence matrix of $K_4$ such that it is $\begin{bmatrix}
0 & 1 & 0\\
1 & 0 & 1
\end{bmatrix}$-free. This can be done for example as follows:
%
%\[
%M=
%\begin{bmatrix}
%1 & 1 & 0 & 0\\
%1 & 0 & 1 & 0\\
%1 & 0 & 0 & 1\\
%0 & 0 & 1 & 1\\
%0 & 1 & 0 & 1\\
%0 & 1 & 1 & 0
%\end{bmatrix}
%\]
\[
M=
\begin{bmatrix}
1 & 1 & 1 & 0 & 0 & 0\\
1 & 0 & 0 & 0 & 1 & 1\\
0 & 1 & 0 & 1 & 0 & 1\\
0 & 0 & 1 & 1 & 1 & 0
\end{bmatrix}\qedhere
\]
\end{proof}

\begin{proof}[Proof of Theorem \ref{thm:not2col}.]
We use the construction of \cite{indec} using the terminology of \cite{abab}.

Let $T(a,b)$ denote a full $a$-ary tree of depth $b-1$. That is, a tree in which every internal vertex has $a$ children and every leaf is at distance $b-1$ from the root of the tree (i.e., the path connecting the root to the leaf contains $b$ vertices). See Figure \ref{fig:t23}.

\begin{figure}
	\begin{center}
		\includegraphics[height=4cm]{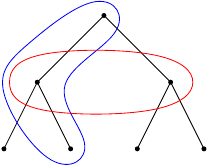}
		\caption{The tree $T(2,3)$ and one horizontal and one vertical hyperedge of $H(2,3)$.}
		\label{fig:t23}
	\end{center}
\end{figure}

Let the hypergraph $H(a,b)$ be defined as follows. Its vertex set is the vertex set of $T(a,b)$, the set of children of each internal vertex is a {\em horizontal} hyperedge of size $a$ and the set of vertices of every path from the root to a leaf is a {\em vertical} hyperedge of size $b$. 

It is easy to see that in every coloring of the vertices of $H(a,b)$ with two colors there is either a monochromatic horizontal hyperedge (of size $a$) or a monochromatic vertical hyperedge (of size $b$). Therefore, $H_2 := H(m,m)$ is an $m$-uniform non-$2$-colorable hypergraph.

By Theorem \ref{thm:matrixequi} it is enough to give an ordering of the rows and columns of the incidence matrix of $H(a,b)$ such that it is $\mabab$-free.

Visualize the underlying tree $T(a,b)$ the usual way: the root vertex is at the top, the children of every vertex are below it and connected with straight segments and vertices that are at the same distance from the root are having the same $y$-coordinate. See Figure \ref{fig:t23}.

Now order the vertices from top to bottom primarily and from left to right secondarily.
The order of the hyperedges is as follows: first, the vertical hyperedges are ordered according to the left to right order of their leaves. Second, the horizontal hyperedges are ordered from top to bottom primarily and from left to right secondarily. Finally, these two are combed together such that a horizontal hyperedge comes as late as possible maintaining that it comes before the first occurrence of any of its vertices in a vertical hyperedge.

We draw the incidence matrices using these orderings of the vertices and hyperedges. For example for $H(2,3)$ we get the following incidence matrix (1's of the horizontal hyperedges are emphasized by having red color while 1's of the vertical hyperedges by having blue color):
\[
M(2,3)=
\begin{bmatrix}
0 & 0 & \blue1 & \blue1 & 0 & \blue1 & \blue1\\
\red1 & 0 & \blue1 & \blue1 & 0 & 0 & 0\\
\red1 & 0 & 0 & 0 & 0 & \blue1 & \blue1\\
0 & \red1 & \blue1 & 0 & 0 & 0 & 0\\
0 & \red1 & 0 & \blue1 & 0 & 0 & 0\\
0 & 0 & 0 & 0 & \red1 & \blue1 & 0\\
0 & 0 & 0 & 0 & \red1 & 0 & \blue1
\end{bmatrix}.
\]

Now we prove that such an incidence matrix is indeed $\mabab$-free. Assume on the contrary that $M(a,b)$ does contain $\mabab$. We distinguish several cases depending on which of its 1 entries are red or blue. In an occurence of $\mabab$ we refer to the corresponding two vertices and three columns as $a,b$ and $A,B,C$, according to the orders defined above.

\begin{enumerate}
	\item $\begin{bmatrix}
	0 & \red1 & 0\\
	\red1 & 0 & 1
	\end{bmatrix}$  or
	$\begin{bmatrix}
	0 & 1 & 0\\
	\red1 & 0 & \red1
	\end{bmatrix}$, i.e., $A$ and at least one of $B$ and $C$ are horizontal:
	
	The red 1 enties are in a strictly descending position in the incidence matrix due to the ordering of the vertices and the horizontal hyperedges. This contradicts these two cases.
		
	\item $\begin{bmatrix}
	0 & \blue1 & 0\\
	\red1 & 0 & \blue1
	\end{bmatrix}$, i.e., $A$ is horizontal and $B,C$ are vertical:

	As the horizontal hyperdge $A$ comes as late as possible maintaining that it comes before the first occurrence of any of its vertices in a vertical hyperedge, there must be a vertical hyperedge $D$ after $A$ and before $B$ containing $b$.	By the ordering of the vertical hyperedges, if two columns contain a blue 1 entry in the same row then all the columns between them that correspond to a vertical hyperedge must also contain a (blue) 1 entry. Thus if $D,C$ contains $b$ then $B$ must also contain $b$, contradicting this case.	
				
	\item $\begin{bmatrix}
		0 & \red1 & 0\\
		\blue1 & 0 & 1
	\end{bmatrix}$, i.e., $A$ is vertical and $B$ is horizontal ($C$ is arbitrary): 
	
	By the definition of the orderings, the unique horizontal hyperedge $D$ containing $b$ must be earlier in the order than $A$. As $B$ contains $a$, a vertex earlier than $b$, $B$ must be earlier than $D$. Thus $B$ is earlier than $D$, which in turn is earlier than $A$, contradicting that $B$ is actually later than $A$.

	\item $\begin{bmatrix}
	0 & 1 & 0\\
	\blue1 & 0 & \red1
	\end{bmatrix}$, i.e., $A$ is vertical, $C$ is horizontal ($B$ is arbitrary):

	As $A$ and $C$ both contain $b$, $C$ must be earlier than $B$ according to the way we combed together horizontal and vertical hyperedges in our ordering. This contradicts this case where $C$ is later than $A$.
	
	\item $\begin{bmatrix}
	0 & \blue1 & 0\\
	\blue1 & 0 & \blue1
	\end{bmatrix}$, i.e., $A,B,C$ are all vertical:
	
	The same way as in Case 2, by the ordering of the vertical hyperedges, if two columns contain a blue 1 entry in the same row then all the columns between them that correspond to a vertical hyperedge must also contain a (blue) 1 entry. Thus if $A,C$ contains $b$ then $B$ must also contain $b$, contradicting this case.
\end{enumerate}
The above cases cover all possibilities, all leading to contradiction and thus finishing the proof.
\end{proof}

\begin{proof}[Proof of Theorem \ref{thm:ext}]
	First, we consider $(AB)^{t/2}$-free hypergraphs.
	For the lower bound split the ordered set of $n$ vertices into $t-1$ intervals of almost equal length. Let $H$ be the hypergraph whose hyperedges are the unions of prefixes of these intervals. It is easy to see that this $H$ is $(AB)^{t/2}$-free. Further, it has $\approx (\frac{n}{t-1})^{t-1}=\Omega(n^{t-1})$ hyperedges, as required.	
	For the upper bound notice that the VC-dimension of an $(AB)^{t/2}$-free hypergraph is at most $t-1$ and then by the Sauer-Shelah lemma it has $O(n^{t-1})$ hyperedges, finishing the proof.
	
	Now, we consider dual-$(AB)^{t/2}$-free hypergraphs.
	By Lemmas \ref{lem:parabolas2ABAB} and \ref{lem:ABAB2parabolas} these are exactly the hypergraphs that can be represented by a family $\C$ of upwards regions whose boundaries form a family of $(t-2)$-intersecting $x$-monotone bi-infinite curves in the plane and a finite point set $S$ such that $H(\C,S)$ is isomorphic to $H$ (the vertices correspond to curves and for each point of the plane the curves below a given point define a hyperedge).
	But $n$ such curves have at most $(t-2)\binom n2/2\le tn^2$ intersection points, so by Euler's formula they divide the plane into $O(tn^2)$ regions; this proves the upper bound for the number of edges.
	%That this is sharp follows from a construction with $\Omega(tn^2)$ regions where no two regions contain the same set of curves under them.
	
	\begin{figure}
		\begin{center}
			\includegraphics[width=15cm]{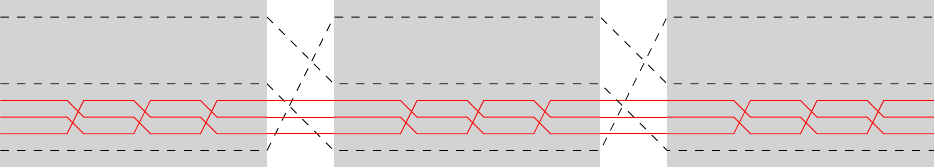}
			\caption{A construction with $n-t'=t'=3$, the non-dashed curves are the $A_i$'s while the dashed curves are the $B_i$'s. The gray areas are the $t'$ parts defined in the proof.}
			\label{fig:dualconstr}
		\end{center}
	\end{figure}

	For the lower bound we define a $(t-2)$-intersecting family of $n$ curves such that the upwards regions above them will be as required. Let $t'=\lfloor (t-2)/2 \rfloor$.
	We name these curves $A_1,\ldots,A_{n-t'},$ $B_1,\ldots,B_{t'}$.
	Always one of the $B_i$ will be the lowest curve (when the curves are intersected by a vertical line).
	First, $B_1$ will be the bottom one, then $B_2$, etc, finally $B_t$.
	This way we define $t'$ vertical strip parts in the plane.
	At each part, all but one of the curves $B_j$ are above every $A_i$.
	This ensures that at each part when we consider a region from between the $A_i$'s, then the only $B_j$ in the respective hyperedge will be the bottommost.
	So hyperedges from different parts containing a non-empty proper subset of the $A_i$'s always differ.
	
	Now we describe how the $A_i$'s in each of the $t'$ parts looks like. See Figure \ref{fig:dualconstr}.
	The curves $A_i$ start in any order, and then the bottommost curve (from the $A_i$'s) goes to the top (of the $A_i$'s), $t'$ times, i.e., each $A_i$ goes to the top exactly once.
	This way each circular interval of the initial order of the $A_i$'s is represented by a region.	
	It is clear from the construction that any two $A_i$'s intersect at most $2$ times in each part, altogether at most $2t'$ times. Two $B_j$'s, or an $A_i$ and a $B_j$ intersect at most twice. Thus, the family of curves is $(t-2)$-intersecting.
	As we get any circular interval of the $A_i$'s $t'$ times as a hyperedge, each time with a different $B_j$, we have at least $t'\binom {n-t'-1}{2}$ different hyperedges. As $t'=\lfloor (t-2)/2 \rfloor$ and $n\ge t>2t'$, this is indeed $\Omega(tn^2)$, where the hidden constant does not depend on $t$.
\end{proof}

\section{Shallow hitting sets and shrinkability}\label{sec:shrinkability}

Another notion that was studied for ABA-free hypergraphs and its relatives, is the existence of a $c$-shallow hitting set for some $c$, that is, the existence of a subset of the vertices of a hypergraph $H$, $V'\subset V(H)$, such that every hyperedge contains at least $1$ and at most $c$ vertices of $V'$. It was shown in \cite{abafree} that ABA-free hypergraphs admit $2$-shallow hitting sets (by finding unskippable vertices; see Section \ref{sec:shrinkability}). Note that the existence of a $c$-shallow hitting set provides a proper $2$-coloring of the hyperedges of size at least $c+1$ (indeed, $V'$ gets the first color, the remaining vertices the second color). As such a coloring does not always exist for ABAB-free \cite{abab} and dual-ABAB-free (Theorem \ref{thm:not2col}), and therefore also for longer-alternation-free hypergraphs, we get that such hypergraphs also cannot always have shallow hitting sets.

\bigskip
For ABA-free hypergraphs it was shown in \cite{abafree} that every hyperedge contains a so-called \emph{unskippable vertex}, i.e., a vertex which we can add as a singleton hyperedge to the hypergraph so that in remains ABA-free.
In case of ABAB-free hypergraphs, adding a singleton hyperedge can never create an ABAB pattern.
However, in \cite{abab} it was shown that every hyperedge contains a so-called \emph{unsplittable pair} of vertices, i.e., two vertices which we can add as a hyperedge of size two to the hypergraph so that in remains ABAB-free.
One might suspect that similar statements also hold for ABABA-free hypergraphs, however, this is not the case.
Below we provide some counterexamples to such statements, all of which were found by a straightforward computer search (we used SAT solvers in Python).

First, we give an ABABA-free hypergraph on 9 vertices.
Below we present the transpose of the incidence matrix, so vertices correspond to columns and hyperedges correspond to rows.
By Theorem \ref{thm:matrixequi}, thus these matrices are free of 
$X_5^T=\begin{bmatrix}
	0 & 1 & 0 & 1 & 0 \\
	1 & 0 & 1 & 0 & 1
\end{bmatrix}$
and of
$(X_5')^T=\begin{bmatrix}
	1 & 0 & 1 & 0 & 1 \\
	0 & 1 & 0 & 1 & 0
\end{bmatrix}$
and we claim that the first hyperedge, $\{2,4,6,8\}$, has no unsplittable pair.
This means that for every pair of vertices from $\{2,4,6,8\}$, there is a hyperedge witnessing its splittability, i.e., adding that pair as a hyperedge would create one of the above two matrices.
In the below matrix, for each hyperedge we marked for which vertex pair it is a witness.
For example, the second hyperedge shows that the vertex pair $\{2,4\}$ cannot be added as a hyperedge, as then the second hyperedge and the new hyperedge would form an ABABA-sequence on vertices $(1,2,3,4,6)$. 

\[
\begin{array}{l|ccccccccc}
	~ & 1 & 2 & 3 & 4 & 5 & 6 & 7 & 8 & 9 \\
	\hline
	\text{first} & 0 & 1 & 0 & 1 & 0 & 1 & 0 & 1 & 0 \\
	(2,4) & 1 & 0 & 1 & 0 & 0 & 1 & 0 & 0 & 0 \\
	(2,6) & 1 & 0 & 0 & 1 & 0 & 0 & 1 & 1 & 0 \\
	(2,8) & 1 & 0 & 0 & 1 & 0 & 0 & 0 & 0 & 1 \\
	(4,6) & 1 & 1 & 1 & 0 & 1 & 0 & 0 & 1 & 0 \\
	(4,8) & 1 & 0 & 0 & 0 & 0 & 1 & 0 & 0 & 1 \\
	(6,8) & 0 & 1 & 0 & 0 & 0 & 0 & 1 & 0 & 1
\end{array}
\]

The below, even simpler example shows that we cannot even hope to add an `unsplittable triple' to an ABABA-free hypergraph.

\[
\begin{array}{l|ccccccccc}
	~ & 1 & 2 & 3 & 4 & 5 & 6 & 7 & 8 & 9 \\
	\hline
	\text{first} & 0 & 1 & 0 & 1 & 0 & 1 & 0 & 1 & 0 \\
	(2,4,6) & 1 & 0 & 0 & 0 & 1 & 0 & 0 & 1 & 0 \\
	(2,4,8) & 1 & 0 & 0 & 0 & 0 & 1 & 0 & 0 & 1 \\
	(2,6,8) & 1 & 0 & 0 & 1 & 0 & 0 & 0 & 0 & 1 \\
	(4,6,8) & 0 & 1 & 0 & 0 & 1 & 0 & 0 & 0 & 1 
\end{array}
\]

Finally, the next example shows that we cannot always add an `unsplittable triple' to an ABABAB-free hypergraph.

\[
\begin{array}{l|cccccccccc}
	~ & 1 & 2 & 3 & 4 & 5 & 6 & 7 & 8 & 9 & 10 \\
	\hline
	\text{first} & 0 & 1 & 0 & 1 & 0 & 1 & 0 & 1 & 0 & 1 \\
    (2,4,6) & 0 & 0 & 1 & 0 & 1 & 0 & 0 & 0 & 0 & 1 \\
(2,4,8) & 0 & 0 & 1 & 0 & 1 & 0 & 0 & 0 & 0 & 1 \\
(2,4,10) & 1 & 0 & 1 & 0 & 0 & 0 & 0 & 1 & 0 & 0 \\
(2,6,8) & 0 & 0 & 0 & 1 & 0 & 0 & 1 & 0 & 1 & 0 \\
(2,6,10) & 1 & 0 & 1 & 0 & 0 & 0 & 0 & 1 & 0 & 0 \\
(2,8,10) & 1 & 0 & 0 & 0 & 0 & 1 & 0 & 0 & 1 & 0 \\
(4,6,8) & 0 & 1 & 0 & 0 & 1 & 0 & 1 & 0 & 0 & 0 \\
(4,6,10) & 0 & 1 & 0 & 0 & 1 & 0 & 1 & 0 & 0 & 0 \\
(4,8,10) & 0 & 1 & 1 & 0 & 0 & 1 & 0 & 0 & 1 & 0 \\
(6,8,10) & 0 & 0 & 0 & 1 & 0 & 0 & 1 & 0 & 1 & 0
\end{array}
\]

\section{Discussion and open problems}\label{sec:discussion}

We have considered certain alternation-free hypergraphs (ABA-free, ABAB-free, etc.) and their duals. In particular we gave a new lower bound on $\chi_m$ of dual-ABAB-free hypergraphs which implies that $\chi_m$ is $3$ or $4$ for this family. It remains an open problem to determine which of these two values is the right one. Also, as we mentioned earlier, dual-ABABA-free hypergraphs and in general dual-$(AB)^{t/2}$-free hypergraphs with $t\ge 5$ were not yet investigated and it is not even known if $\chi_m$ is finite for them. Our equivalent definitions using their incidence matrices may be useful to construct hypergraphs showing that $\chi_m=\infty$, provided they exist.

\smallskip

%There are various further problems that naturally arise when we consider some fixed families of hypergraphs.
%EZT A RÉSZT ÁTÍRNI
%First, what is the maximum size of a hypergraph from a given family? Second, what is the maximum size of such a $k$-uniform hypergraph for fixed $k$? What is the size of a maximum size containment-free (also called Sperner or antichain) hypergraph from the family? What about enumeration problems, that is, on $n$ vertices how many members does a given family have? 
%These are all directions for possible further research. We give a relatively good answer for the first proposed problems in the primal (i.e., non-dual) setting by showing the 

Finally, from a computational perspective, how hard is it to decide if a hypergraph belongs to a certain family? Is there a polynomial time algorithm to decide if a hypergraph on an unordered vertex set is ABA-free, ABAB-free, etc.?
We will discuss these questions in an upcoming paper, joint with G\'abor Dam\'asdi and Karamjeet Singh, whom we thank for discussions, especially about Theorem \ref{thm:ext}.

%\smallskip 
%\noindent {\bf Acknowledgement}
%We thank G\'abor Dam\'asdi and Karamjeet Singh for discussions about Theorem \ref{thm:ext}.

\bibliographystyle{plainurl}
\bibliography{psdisk}
\end{document}